\documentclass[11pt, oneside]{article}  
\usepackage{geometry}   
\usepackage{graphicx}				
\usepackage{epsfig}	
								
\usepackage{tikz}
\usepackage{amssymb, amsmath}
\newcommand{\vx}{{\bf x}}
\newcommand{\bj}{{\bf j}}
\newcommand{\bv}{{\bf v}}
\newcommand{\bn}{{\bf n}}
\newcommand{\f}{{f}}
\newcommand{\du}{{\partial_u}}
\newcommand{\dv}{{\partial_v}}

\newcommand{\bbR}{{\mathbb R}}
\newcommand{\bbC}{{\mathbb C}}

\newcommand{\calM}{{\mathcal M}}
\newcommand{\calH}{{\mathcal H}}

\newtheorem{definition}{Definition}
\newtheorem{remark}{Remark}
\newtheorem{lemma}{Lemma}
\newenvironment{proof}[1][Proof]{\begin{trivlist}
\item[\hskip \labelsep {\bfseries #1}]}{\end{trivlist}}

\title{Pseudo-spectral methods for the Laplace-Beltrami equation
and the Hodge decomposition on surfaces of genus one\footnote{This  work  was  supported  in  part  by  the  Applied  Mathematical  Sciences  Program  of  the
U.S. Department of Energy under contract DEFGO288ER25053 and by the Office of the Assistant
Secretary  of  Defense  for  Research  and  Engineering  and  AFOSR  under  NSSEFF  program  award
FA9550-10-1-0180}}
\author{Lise-Marie Imbert-Gerard 
\footnote{Courant Institute, New York University, New
    York, NY. Email: {\tt imbertgerard@cims.nyu.edu}}, \, and
Leslie Greengard\footnote{Courant Institute, New York University,
  and Simons Center for Data Analysis, Simons Foundation, New York, NY. Email:
  {\tt greengard@cims.nyu.edu}}
}
\date{\today}

\begin{document}

\maketitle
\begin{abstract}
The inversion of the Laplace-Beltrami operator and the computation of 
the Hodge decomposition of a tangential vector field on smooth surfaces
arise as computational tasks in many areas of science,
from computer graphics to machine learning to computational
physics. Here, we present a high-order accurate pseudo-spectral
approach, applicable to closed surfaces of genus one 
in three dimensional space, with a view toward applications in plasma
physics and fluid dynamics. 
\end{abstract}


\section{Introduction}

Partial differential equations on smooth manifolds
\cite{Jost,ROSENBERG97,Schwarz}
arise in many areas of science, including
graphics and image analysis
\cite{Chen2015,Fisher2007,GuYau02,CGF2025,Petronetto2010,Sapiro,Xu2009391},
machine learning \cite{BELKIN2008,Coifman2005}, 
plasma
physics and fluid dynamics 
\cite{BOOZER1999,PLASMACOORDS,Duduchava2006,GARABEDIAN,Marsh}, 
and electromagnetic theory
\cite{Chernokozhin2013,Colton,EG10,EGO13,EGO15,GROSSKOTIUGA,Nedelec}.
In some cases, the problem of interest concerns a
physical process such as surface growth or surface diffusion 
\cite{Ruuth2008,Schwartz2005}.
In others, the primary concern is with the representation of 
vector fields on the surface, such as currents which give rise to
electromagnetic fields in the ambient space. In the latter context,
it is known from Hodge theory \cite{Schwarz,Warner}
that one can express any vector field $\bj$ on a smooth
closed multiply connected surface $\Gamma$ in $\bbR^3$ in terms 
of an orthogonal decomposition of the form:
\begin{equation}
\label{helmdecompmc}
\bj =  \nabla_{\Gamma} \alpha + \bn \times \nabla_{\Gamma} \beta + \bj_H
\end{equation}
for some scalar functions $\alpha, \beta$,
where $\bn$ is the surface normal,
$\nabla_{\Gamma}$ denotes the surface gradient on $\Gamma$, and
$\bj_H$  is both surface divergence-free and surface curl-free, namely
\begin{equation} \label{harmdef}
 Div \, \bj_H =  \nabla_{\Gamma} \cdot \bj_H = 0, \quad
 Curl \, \bj_H = \nabla_{\Gamma} \cdot (\bn \times \bj_H) = 0. 
\end{equation}
These operators are defined formally in terms of the metric
tensor in the next section. For the moment, we simply note that
$Curl \, \bj$ is defined as a scalar field, corresponding to the
normal component of the three dimensional curl
of a suitable extension of the tangential vector field $\bj$ 
in a neighborhood of the
surface \cite{Colton,Nedelec}.
Vector fields satisfying both equations in \eqref{harmdef} are called 
{\em harmonic vector fields} 
and, from Hodge theory, the space of such functions is of dimension
exactly $2\mathfrak g$ for a surface of genus $\mathfrak g$ \cite{Schwarz,Warner}.
We also note (for later reference) that
$Curl \, (\nabla_{\Gamma} \alpha) =0$ and 
$Div \, (\bn \times \nabla_{\Gamma} \beta) =0$.

Despite their importance, differential equations on manifolds have 
received much less attention
from the scientific computing community than
boundary value problems in Euclidean space. 
For some sense of the available literature, we refer the reader to 
\cite{Arnold2010,Dziuk2013,Olshanskii2009} 
and the references therein for a discussion of 
finite element analysis, 
\cite{ADALSTEINSSON1997,BertalmioOsher,HUISKAMP1991,Macdonald2011,Ruuth2008}
for a discussion of finite difference methods, and 
\cite{KN2014} for a discussion of integral equation methods. 

Here, we restrict our attention to smooth, closed
surfaces $\Gamma$ of genus one embedded in $\bbR^3$, primarily because 
of their importance in plasma
physics and stellarator design \cite{garabedian_stell}. More 
precisely, we assume that $\Gamma$ 
is a smooth toroidal surface:
\begin{equation}\label{eq:defX}
\vx(u,v) = 
( x(u,v), y(u,v), z(u,v) ),
\end{equation}
where $\vx(u,v)$ is a doubly periodic
function of $(u,v)\in [0,2\pi]^2$. 
In the case of a surface of revolution,
this takes the special form
\begin{equation}\label{eq:SR}
\vx(u,v) = (r(u)\cos v, r(u)\sin v, z(u)).
\end{equation}
We will refer, in either setting, to 
$u$ as the {\em poloidal} angle and to $v$ as the {\em toroidal} angle.

A smooth real-valued scalar function
$f(u,v)$ defined on $\Gamma$ is periodic, and it will
be convenient to write $f$ in the form of a Fourier series:
\[
f(u,v) 
= \sum_{m=-\infty}^{\infty} \sum_{n=-\infty}^{\infty}
f_{mn} e^{\imath (mu+nv)}.
\]
A smooth, real-valued tangential vector field 
$\mathbf F$ on the
surface defined by \eqref{eq:defX} is generally expressed in
terms of the standard, contravariant components in the standard, covariant basis:
\begin{equation}\label{eq:Fbasis}
{\mathbf F}(u,v) = F^1(u,v)\du \vx(u,v) + F^2(u,v)\dv \vx(u,v).
\end{equation}
For simplicity, we will often write ${\mathbf F} = (F^1,F^2)$.
Since 
$\mathbf F : [0,2\pi[^2 \rightarrow \bbR^3$ 
is periodic, it can be
represented as a Fourier series of the form
\[
\mathbf F(u,v) 
= \sum_{m=-\infty}^{\infty} \sum_{n=-\infty}^{\infty}
\bigg(F^1_{mn} \du \vx(u,v) + F^2_{mn}\dv \vx(u,v)\bigg) e^{\imath (mu+nv)}.
\]
The space of periodic tangential vector fields 
will be denoted $\mathcal V_{\Gamma}$ and
the subspace of $\mathcal V_{\Gamma}$ consisting of $k$-times 
differentiable functions will be denoted by 
$\mathcal V_{\Gamma}^k$. Similarly, the space of periodic 
scalar functions will be denoted by ${\mathcal S}_\Gamma$,
and its subspace of $k$-times differentiable functions by 
${\mathcal S}^k_\Gamma$.

In this paper, we construct efficient, high-order pseudospectral methods for 
evaluating the surface divergence, gradient and curl
(section \ref{diffgeo}),
inverting the Laplace-Beltrami operator 
(section \ref{surflap}),
determining a basis for surface harmonic vector fields
(section \ref{harmvec}), and 
computing the Hodge decomposition
(section \ref{hodge}).
Section \ref{numex} contains some illustrative numerical 
examples and section \ref{discussion} contains some concluding
remarks.

\section{Mathematical preliminaries} \label{diffgeo}

This section is dedicated to the definition and discretization of  $L^2$ inner products of scalar and vector-valued functions, and partial differential operators on a surface of genus $1$ in $\mathbb R^3$.
\subsection{Surface operators}
For a surface $\Gamma$ defined by a parametrization \eqref{eq:defX}, we will denote the 
metric tensor by 
\[
 G = \begin{pmatrix}
G_{uu} & G_{uv} \\ G_{uv} & G_{vv}
\end{pmatrix}
\] 
where
\[
G_{uu} = \du \vx\cdot\du \vx, \quad
G_{uv} = \du \vx\cdot\dv \vx,\quad \text{and}\quad
G_{vv} = \dv \vx\cdot\dv \vx.
\] 
and the determinant of $G$ by $g$. By the Cauchy-Schwarz inequality, it 
is clear that $g\geq 0$. We assume $g>0$, so that 
$(\du\vx,\dv\vx)$ is a basis for the tangent space to the surface at every 
point under the parametrization \eqref{eq:defX}.
The induced $L_2$ inner product for tangential vector fields $\bv_1,\bv_2$
is
\begin{equation} \label{l2vecdef}
\langle \bv_1,\bv_2 \rangle_\Gamma = 
\int_\Gamma \bv_1(\vx(u,v)) \cdot \bv_2(\vx(u,v)) dA(u,v) = 
\int_0^{2\pi} \int_0^{2\pi}
 \bv_1(u,v) \cdot \bv_2(u,v) \sqrt{g(u,v)} \, du \, dv .
\end{equation}
The induced $L_2$ inner product for scalar fields $\phi,\psi$
is
\begin{equation} \label{l2def}
\langle \phi,\psi \rangle_\Gamma = 
\int_\Gamma \phi(\vx(u,v)) \, \psi(\vx(u,v)) dA(u,v) = 
\int_0^{2\pi} \int_0^{2\pi}
 \phi(u,v) \, \psi(u,v) \sqrt{g(u,v)} \, du \, dv 
\end{equation}
The surface normal is given by the unit vector
$\mathbf n = \frac{1}{\sqrt{g}}(\du \vx \times \dv \vx)$.

\begin{definition} \label{defops}
Assume the surface $\Gamma$ is such that $G_{uu}$, $G_{uv}$ and $G_{vv}$ are in $\mathcal S_\Gamma^1$, and that $g\neq 0$.
Let $f(u,v)$ denote a scalar function and let 
${\mathbf F}(u,v)$ denote a tangential vector field on the surface $\Gamma$.
Then the surface gradient, divergence, curl and Laplace-Beltrami
operator are given by \cite{Colton,Nedelec}:
\begin{align*}
Grad \, f &= \nabla_\Gamma f
   =  \left[\frac{G_{vv}}{g} \du f - \frac{G_{uv}}{g} \dv f \right] \du \vx 
+ \left[-\frac{G_{uv}}{g} \du f + \frac{G_{uu}}{g} \dv f \right] \dv \vx ,
\forall f\in{\mathcal S}^1_\Gamma,\\
Div \, \mathbf F &= \nabla_\Gamma \cdot \mathbf F = 
\frac {1}{\sqrt g} \big[\du\left(\sqrt{ g}F^1\right) +\dv\left(\sqrt{ g}F^2\right) \big] \, ,
\forall \mathbf F\in{\mathcal V}^1_\Gamma , \\
Curl \, \mathbf F &= - \nabla_\Gamma \cdot (\mathbf n \times \mathbf F) =
\frac 1{\sqrt g}\big(\du\left(G_{uv}F^1 + G_{vv}F^2\right) 
 -\dv\left( G_{uu}F^1 + G_{uv}F^2\right) \big) \, ,
 \forall \mathbf F\in{\mathcal V}^1_\Gamma, \\
\Delta_\Gamma f &= \nabla_\Gamma \cdot ( \nabla_\Gamma f) 
=\frac{1}{\sqrt g} \left[
 \du \left(\frac{G_{vv}}{\sqrt g}\du f
          -\frac{G_{uv}}{\sqrt g}\dv f\right) 
+\dv \left(\frac{G_{uu}}{\sqrt g}\dv f
          -\frac{G_{uv}}{\sqrt g}\du f\right) \right]\, ,
\forall f\in{\mathcal S}^2_\Gamma,
\end{align*}
respectively.
\end{definition}

 
\subsection{Discretization} \label{sec:discr}
In order to develop high-order
pseudospectral methods, we assume that scalar 
functions, such as $f$ or $F^1$, are
discretized on a uniform grid with $N^2$ points:
\[ \vec{f} = \{ f(u_k,v_l) \}, \ 
\vec{F}^1 = \left\{ F^1(u_k,v_l) \right\}, 
\]
where
\[ h = \frac{2\pi}{N}, \quad
u_k = k h, \quad
v_l = l h,\quad k,l=1,\dots,N.
\]
We approximate the Fourier coefficients
\[
\left[\mathcal F(f)\right]_{mn} \equiv
\frac{1}{(2\pi)^2}
\int_0^{2\pi}
\int_0^{2\pi}
f(u,v)e^{-\imath(mu+nv)}du\,dv
\]
by the discrete Fourier transform 
$\widehat{f} = \mathcal F_{h}\left(\vec{f}\,\right)$:
\[
\widehat{f}_{mn} =
\left(\frac{h}{2\pi}\right)^2\sum_{k=1}^N \sum_{l=1}^N 
f(u_k,v_l)e^{-\imath(mu_k+nv_l)}\, .
\]
With a slight abuse of notation, we assume that 
the set of discrete values on the surface, namely $\vec{f} \in {\bbR}^{N^2}$, and
the corresponding discrete Fourier transform, 
$\widehat{f} \in \bbC^{N^2}$, 
have been ``unrolled" and can be considered
column vectors of length $N^2$. Likewise, we will view
$\left(\vec{F}^1,  \vec{F}^2 \right)$ or
$\left( \vec{F}^1, \vec{F}^2 \right)^T$
as column or row vectors in $\bbR^{2N^2}$.

\begin{definition} \label{defops2}
Given the Fourier coefficient vector $\widehat{f}$, we will denote by
$\mathcal D_{\sigma(n,m)}$ the (Fourier space)
diagonal operator such that 
\[ \left[\mathcal D_{\sigma(n,m)}\widehat f\,\right]_{mn} = \sigma(n,m) \, \widehat f_{mn}\, . \]
Thus, $ \left[\mathcal D_{im}\widehat f\,\right]_{mn} = im \, \widehat f_{mn}$, etc.

Given the vector $\vec{\f}$, we will denote by
$D_{ug}$ and $D_{vg}$ the (physical space) diagonal operators such that 
\[ \left[D_{ug}\vec{f}\,\right](u_k,v_l) =
\frac{\partial_u g}{2g}(u_k,v_l) \f(u_k,v_l),\quad
 \left[D_{vg}\vec{f}\,\right](u_k,v_l) =
\frac{\partial_v g}{2g}(u_k,v_l) \f(u_k,v_l).
\]

We will denote by $D_{uu}$, $D_{uug}$, $D_{uv}$, $D_{uvg}$, 
$D_{vv}$, and $D_{vvg}$,
the diagonal operators such that 
\[  \left[D_{uu} \vec{f}\,\right](u_k,v_l) =
\frac{G_{uu}}{g}(u_k,v_l) \f(u_k,v_l), \quad
 \left[ D_{uug} \vec{f}\,\right](u_k,v_l) =
\frac{G_{uu}}{\sqrt{g}}(u_k,v_l) \f(u_k,v_l),
\]
\[  \left[D_{uv} \vec{f}\,\right](u_k,v_l) =
-\frac{G_{uv}}{g}(u_k,v_l) \f(u_k,v_l), \quad
  \left[D_{uvg} \vec{f}\,\right](u_k,v_l) =
\frac{G_{uv}}{\sqrt{g}}(u_k,v_l) \f(u_k,v_l),
\]
\[  \left[D_{vv} \vec{f}\,\right](u_k,v_l) =
\frac{G_{vv}}{g}(u_k,v_l) \f(u_k,v_l), \quad
  \left[D_{vvg} \vec{f}\,\right](u_k,v_l) =
\frac{G_{vv}}{\sqrt{g}}(u_k,v_l) \f(u_k,v_l).
\]

We will denote by $D_{{u,uv}}$, $D_{{v,uv}}$, 
$D_{{v,uu}}$, and $D_{{u,vv}}$, 
the (physical space) diagonal operators such that 
\[  \left[D_{{u,uv}} \vec{f}\,\right](u_k,v_l) =
\frac{\partial_u G_{uv}}{\sqrt{g}}(u_k,v_l) \f(u_k,v_l),
\quad
  \left[D_{{v,uv}} \vec{f}\,\right](u_k,v_l) =
\frac{\partial_v G_{uv}}{\sqrt{g}}(u_k,v_l) \f(u_k,v_l),
\]
\[  \left[D_{{v,uu}} \vec{f}\,\right](u_k,v_l) =
\frac{\partial_v G_{uu}}{\sqrt{g}}(u_k,v_l) \f(u_k,v_l),
\quad
  \left[D_{{u,vv}} \vec{f}\,\right](u_k,v_l) =
\frac{\partial_u G_{vv}}{\sqrt{g}}(u_k,v_l) \f(u_k,v_l),
\]
\end{definition}

\begin{lemma} \label{opdeflemma}
Let $\vec{f}$ denote a discretized scalar function on a surface $\Gamma$
and let $\vec{\mathbf F} = \left(\vec{F^1},\vec{F^2}\right)^T$ denote a 
discretized tangential vector field in the standard basis. 
Then the surface gradient, divergence, and Laplace-Beltrami
operator are approximated by:
\begin{align*}
Grad_h \, \vec{f} &= \nabla_{\Gamma,h} \vec{f}
   = 
\begin{pmatrix} D_{vv} & D_{uv} \\ D_{uv} & D_{uu} \end{pmatrix}
\begin{pmatrix} \mathcal F_{h}^\ast & 0 \\ 0 & 
\mathcal F_h^\ast \end{pmatrix} 
\begin{pmatrix} \mathcal D_{im} \\ 
\mathcal D_{in} \end{pmatrix} 
\mathcal F_{h}   \vec{f}  \\
Div_h \, \vec{\mathbf F} &= 
\nabla_{\Gamma,h} \cdot \vec{\mathbf F} = 
 \begin{pmatrix} V_1(G) & V_2(G) \end{pmatrix}
\begin{pmatrix} \vec{F^1} \\ \vec{F^2} \end{pmatrix} \\
Curl_h \, \vec{\mathbf F} &= \begin{pmatrix} C_1(G) & C_2(G) \end{pmatrix}
\begin{pmatrix} \vec{F^1} \\ \vec{F^2} \end{pmatrix} \\
\Delta_{\Gamma,h} \vec{f} &= Div_h \, Grad_h \vec{f}.
\end{align*}
where
\begin{align*}
V_1(G) &=  D_{ug} + \mathcal F_h^\ast \mathcal D_{im} \mathcal F_h \\
V_2(G) &= D_{vg} +  \mathcal F_h^\ast \mathcal D_{in} \mathcal F_h \\
C_1(G) &= D_{u,uv} - D_{v,uu} +
D_{uvg} \mathcal F_h^\ast \mathcal D_{im} \mathcal F_h -
D_{uug} \mathcal F_h^\ast \mathcal D_{in} \mathcal F_h,
\\
C_2(G) &= D_{u,vv} - D_{v,uv} +
D_{vvg} \mathcal F_h^\ast \mathcal D_{im} \mathcal F_h -
D_{uvg} \mathcal F_h^\ast \mathcal D_{in} \mathcal F_h .
\end{align*}
Assuming $f$, $\mathbf F \in C^{k}$ and the surface $\Gamma \in C^{k+2}$,
all of the above approximations are $k$-th order accurate.
If $f$, $\mathbf F$ and $\Gamma$ are infinitely differentiable, then 
the convergence rate is superalgebraic.
\end{lemma}

\begin{proof}
The formulas are obtained directly from 
Defintions \ref{defops} and \ref{defops2}. 
The convergence rate is easily established
from the corresponding, well-known properties of the discrete 
Fourier transform.
\end{proof}

Finally, the $L^2$ inner product $\langle f_1,f_2\rangle_\Gamma$ 
is computed by
\[
\langle\vec{f_1},\vec{f_2}\rangle_{\Gamma,h} = \sum_{k=1}^N\sum_{l=1}^N f_1(u_k,v_l)f_2(u_k,v_l) \sqrt{g(u_k,v_l)}.
\]

\section{The Laplace-Beltrami equation} \label{surflap}

Let us now consider the problem of solving the Laplace-Beltrami
equation
\begin{equation}
\label{lapeq}
 \Delta_\Gamma \phi = b
\end{equation}
on a toroidal surface $\Gamma$. 
For this, we first
define $\calM_\Gamma$, the space of mean-zero functions on
$\Gamma$:
\[ \calM_\Gamma = \{ f: \Gamma \rightarrow \bbR \, | \,
\langle f,e \rangle_\Gamma =  0 \},
\]
where $e(u,v)$ is the constant function $e(u,v) = 1$.

It is well-known that \eqref{lapeq} is rank-one deficient, but
invertible as a map from $\calM_\Gamma$ to $\calM_\Gamma$
\cite{Jost,ROSENBERG97}. In order to formulate the
problem in terms of an invertible linear system, we replace
\eqref{lapeq} with
\begin{equation}
\label{lapeqstab}
 \Delta_\Gamma \phi + \langle e,\phi \rangle_\Gamma \, e = b,
\end{equation}
where, again, $e(u,v)=1$. The left-hand side is easily shown to be
an invertible operator on the space of smooth functions. Moreover,
 \eqref{lapeqstab}
has a unique solution if the right-hand side $b \in \calM_\Gamma$. 
To see this, note that 
\[
\langle e,\Delta_\Gamma \phi \rangle_\Gamma + \langle e,\phi \rangle_\Gamma \, \langle e,e \rangle_\Gamma = 
\langle e,b \rangle_\Gamma.
\]
Since $\Delta_\Gamma$ is self-adjoint and annihilates constant functions, 
we have 
\[
0 + \langle e,\phi \rangle_\Gamma \, \langle e,e \rangle_\Gamma = 0,
\]
from which $\langle e,\phi \rangle_\Gamma = 0$. 

It remains to develop a finite-dimensional version of \eqref{lapeqstab}.
For this, we assume $\phi$ and $b$ are discretized as in 
section \ref{sec:discr}, and 
let $c(u,v) = \sqrt{g(u,v)}$ with $\vec{c}$ denoting the corresponding
discrete vector. It is then easy to see that the discretization of 
\eqref{lapeqstab} on a uniform mesh takes the form
\begin{equation} \label{lapeqdisc}
\left( \Delta_{\Gamma,h} + \vec{c} \,\vec{c}^{\, T} \right) 
\vec{\phi} = \vec{b}.
\end{equation}
In order to construct a better-conditioned linear system, we
will use the inverse of the ``flat Laplacian" as a preconditioner
and solve
\begin{equation} \label{lapeqdiscp}
\left( \Delta_{I,h} + \vec{1} \, \vec{1}^{\, T} \right)^{-1} 
\left( \Delta_{\Gamma,h} + \vec{c} \, \vec{c}^{\, T} \right) \vec{\phi} = 
\left( \Delta_{I,h} + \vec{1} \, \vec{1}^{\, T}\right)^{-1} 
\vec{b}.
\end{equation}
Here, $\Delta_{I,h}$ is the discrete Laplace-Beltrami operator defined 
above, but for the flat torus, for which the metric tensor
$G$ is the identity operator. The vector $\vec{1}$ denotes the vector of 
length $N^2$, each of whose entries is one.

\begin{remark}
We note that the operator 
$\left( \Delta_{\Gamma,h} + \vec{c} \, \vec{c}^{\, T} \right)$ can be applied
in $O(N^2 \log N)$ operations since all the operators in 
$Div_h$ and $Grad_h$ are either diagonal or involve the discrete
Fourier transform, which can be computed by means of the fast 
Fourier transform
(FFT).
The preconditioner 
$\left( \Delta_{I,h} + \vec{1} \, \vec{1}^{\, T} \right)^{-1}$ can also be 
applied in $O(N^2 \log N)$ operations since the flat Laplacian
$\Delta_{I,h}$ is diagonalized by the Fourier transform and inversion
corresponds to multiplication by $-1/((n^2+m^2)$ in the Fourier
domain, unless $n=m=0$, in which case we multiply by $-1$.
\end{remark}

\section{Harmonic vector fields} \label{harmvec}

As discussed in the introduction, 
harmonic vector fields are tangential vector fields 
$\mathbf F \in \mathcal V_{\Gamma}$ satisfying the system
  \begin{equation} \label{hvfsystem}
  \begin{array}{r}
  Div (\mathbf F) =  \nabla_\Gamma \cdot {\mathbf F} = 0\, \\
  Curl(\mathbf F) =  \nabla_\Gamma \cdot (\bn \times {\mathbf F}) = 0.
  \end{array}
  \end{equation}
Thus, one method for their computation 
is to find the nullspace of \eqref{hvfsystem},
which we will denote by $\calH(\Gamma)$.
  
\begin{remark}
For a surface $\Gamma$ embedded in $\bbR^3$, if 
$\mathbf F \in \calH(\Gamma)$, then 
$\mathbf n \times \mathbf F$ is in $\calH(\Gamma)$ as well. 
This follows from the fact that
$\mathbf n \times(\mathbf n \times \mathbf F) = -\mathbf F$.
Thus, since $\calH(\Gamma)$ is two-dimensional, we need only find a single 
non-trivial element in the nullspace of 
\eqref{hvfsystem}. A second basis vector is then obtained by simply
computing the cross-product of the first with $\mathbf n$. 
\end{remark}

The finite-dimensional version of \eqref{hvfsystem}, using the 
discretization of section \ref{sec:discr}, is
  \begin{equation} \label{hvfdisc}
A \begin{pmatrix} \vec{F}^1 \\ \vec{F}^2 \end{pmatrix}
=
\left( 
\begin{array}{cc} V_1(G) & V_2(G) \\ [.5em]
C_1(G) & C_2(G) \end{array} \right)
\begin{pmatrix} \vec{F}^1 \\ \vec{F}^2 \end{pmatrix}
= \left( 
\begin{array}{cc} 0 \\ [.5em] 0 \end{array} 
\right).
\end{equation}
This is a real $2N^2 \times 2N^2$ system of equations.

\subsection{Computation of the nullspace}

To find a null vector, we use the method of \cite{nullspace}, which 
we very briefly review here. The method consists of replacing
\eqref{hvfsystem}, which we denote by $A \vec{\mathbf F} = 0$, with
  \begin{equation} \label{hvfdiscp}
  \left( A + R \, S^{\, T} \right) 
\begin{pmatrix} \vec{F}^1 \\ \vec{F}^2 \end{pmatrix}
= A 
\begin{pmatrix} \vec{Q}^1 \\ \vec{Q}^2 \end{pmatrix}
  \end{equation}
where $R,S$ are random matrices in $\bbR^{2N^2 \times 2}$,
$\begin{pmatrix} \vec{Q}^1 \\ \vec{Q}^2 \end{pmatrix}$ 
is a random vector in $\bbR^{2N^2}$, and 
$S^T$ denotes the transpose of $S$.

A probabilistic argument \cite{nullspace} shows that
$[A + R \, S^{\, T}]$ is invertible.
After the solution of \eqref{hvfdiscp}, note that 
$A 
\begin{pmatrix} \vec{Q}^1 \\ \vec{Q}^2 \end{pmatrix}$
is in the range of $A$, as is $A 
\begin{pmatrix} \vec{F}^1 \\ \vec{F}^2 \end{pmatrix}$.
Since, with probability one, $R$ is {\em not} in the range of $A$,
it must follow that $S^T
\begin{pmatrix} \vec{F}^1 \\ \vec{F}^2 \end{pmatrix} = 0$.
Hence, $\begin{pmatrix} \vec{F}^1-\vec{Q}^1  \\
\vec{F}^2-\vec{Q}^2  \end{pmatrix}$
is a non-trivial null vector.

In practice, we will solve \eqref{hvfdiscp}
iteratively, using as a preconditioner the 
``flat" approximation to $A^{-1}$, 
namely
  \begin{equation} \label{hvfprec}
\left( 
\begin{array}{cc} V_1(I) & V_2(I) \\ [.5em]
C_1(I) & C_2(I) \end{array} \right)^{\dagger} =
\left(
\begin{array}{cc}
 \mathcal F_h^\ast \mathcal D_{im} \mathcal F_h &
\mathcal F_h^\ast \mathcal D_{in} \mathcal F_h \\ [.2em]
-\mathcal F_h^\ast \mathcal D_{in} \mathcal F_h 
&
 \mathcal F_h^\ast \mathcal D_{im} \mathcal F_h
\end{array} \right)^{\dagger}.
\end{equation}
Here,
$V_1(I), V_2(I), C_1(I)$ and $C_2(I)$ denote the discrete divergence
and curl components 
from Lemma \ref{opdeflemma}, using the identity
operator in place of the metric $G$.
We use the pseudoinverse symbol here ($\dagger$) to indicate that, 
when $D_{in}$ or $D_{im}$ is to be formally inverted, we regularize
the matrix by setting any zero diagonal elements to $1$.
It is straightforward to verify that 
the preconditioner in \eqref{hvfprec} can be applied in 
$O(N^2 \log N)$ operations using the FFT.

To summarize, a basis
$\{ \mathbf h_1, \mathbf h_2 \}$ for the space of 
harmonic vector fields $\calH(\Gamma)$ on a surface $\Gamma$ of genus one
can be computed 
by solving \eqref{hvfdiscp} and letting
$\mathbf h_1 = \begin{pmatrix} \vec{F}^1-\vec{Q}^1  \\
\vec{F}^2-\vec{Q}^2  \end{pmatrix}$.
We then 
set $\mathbf h_2 = \bn \times \mathbf h_1$.

\begin{remark}
For a surface of revolution parametrized as in \eqref{eq:SR},
the metric tensor takes the simpler form
\[ G 
 = \begin{pmatrix}
r'(u)^2 + z'(u)^2
&0\\
 0
&r(u)^2
\end{pmatrix}.
\]
Since the matrix is diagonal and independent of the parameter $v$, it is easy to find a harmonic vector field. 
Indeed, the divergence and curl operators reduce to
\begin{align*}
 Div (\mathbf F) &= \frac 1{\sqrt g}\du\left(\sqrt{ g}F^1\right) 
 +\dv F^2 , \\
 Curl(\mathbf F) &= \frac 1{\sqrt g}\big(\du\left( G_{vv}F^2\right) 
 - G_{uu}\dv F^1  \big) .
\end{align*}
It is straightforward to verify that 
$\mathbf h_1 = \left(0,\frac{1}{G_{vv}}\right) = \frac{1}{G_{vv}}\dv \vx $ 
is a harmonic vector field, without the need to solve any system of 
equations.
\end{remark}

\section{The Hodge decomposition} \label{hodge}

Suppose now that the method of the previous section
has been used to compute a basis $\{ \mathbf h_1, \mathbf h_2 \}$
for the space of harmonic vector fields
$\calH(\Gamma)$ on a surface of genus one.
Given a tangential vector field $\bj$, we seek its orthogonal
decomposition
\[
\bj = \nabla_{\Gamma} \alpha + \bn \times \nabla_{\Gamma} \beta 
+ \bj_H,
\]
where $\bj_H = c_1 \mathbf h_1 + c_2 \mathbf h_2$,
for unknown scalar functions $\alpha, \beta$ and unknown 
constants $c_1,c_2$.
Taking the surface divergence of either $\bj$ or $\bn \times \bj$,
one obtains the two Laplace-Beltrami equations:
\[ \Delta_\Gamma \alpha = \nabla_\Gamma \cdot \bj, \quad
\Delta_\Gamma \beta =  - \nabla_\Gamma \cdot (\bn \times \bj). \]
The right hand sides can be computed using the methods of 
section \ref{diffgeo}, and these equations can be solved using 
the method of section \ref{surflap}. 
The constants $c_1$ and $c_2$ are easily computed by solving 
the $2 \times 2$ linear system
\[
\begin{pmatrix}
\langle \mathbf h_1, \mathbf h_1 \rangle_\Gamma &
\langle \mathbf h_1, \mathbf h_2 \rangle_\Gamma \\
\langle \mathbf h_2, \mathbf h_1 \rangle_\Gamma &
\langle \mathbf h_2, \mathbf h_2 \rangle_\Gamma \end{pmatrix}
\begin{pmatrix} c_1 \\ c_2 \end{pmatrix}
=
\begin{pmatrix} 
\langle \mathbf h_1, \bj \rangle_\Gamma \\
\langle \mathbf h_2, \bj \rangle_\Gamma 
\end{pmatrix}.
\]

\section{Numerical Examples} \label{numex} 

We illustrate the performance of our methods using a single (non-axisymmetric)
surface.  For this, 
Garabedian coordinates \cite{garabedian_coords,garabedian_stell} 
provide a parametrization of a stellarator geometry - a surface of genus 
$\mathfrak g=1$ with a central curve, referred to as the magnetic axis, 
embedded in the interior. This involves three parameters: the poloidal and 
toroidal angles, as well as a radius-like parameter $s$. The magnetic axis 
is a curve which depends only on the toroidal angle $v$ according to
the equation
\begin{equation}
\left\{
\begin{array}{l}
x(s,u,v) = r_0(v) \cos v , \\
y(s,u,v) = r_0(v) \sin  v , \\
z(s,u,v) = z_0(v),
\end{array}
\right.
\end{equation}
where $r_{0}$ and $z_{0}$ are $2\pi$-periodic functions of $v$.
The {\em outer surface} of the stellarator is defined by a finite number
of trigonometric modes:
\begin{equation}
\left\{
\begin{array}{l}
x(s,u,v) = \sum \Delta_{m,n} \cos ((1-m)u+nv),  \\
y(s,u,v) = \sum \Delta_{m,n} \cos ((1-m)u+nv),  \\
z(s,u,v) = \sum \Delta_{m,n} \sin ((1-m)u+nv), 
\end{array}
\right.
\end{equation}
where $\{ \Delta_{m,n} \}$ are real coefficients.
In between the magnetic axis and the stellarator surface, we have
\begin{equation}
\left\{
\begin{array}{l}
x(s,u,v) = \cos v \Big(r_0(v) + R(s,u,v)\left(\sum \Delta_{m,n} \cos ((1-m)u+nv) - r_0(v)\right)\Big), \\
y(s,u,v) = \sin  v \Big(r_0(v) + R(s,u,v)\left(\sum \Delta_{m,n} \cos ((1-m)u+nv) - r_0(v)\right)\Big), \\
z(s,u,v) = z_0(v) + R(s,u,v)\left(\sum \Delta_{m,n} \sin ((1-m)u+nv) - z_0(v)\right).
\end{array}
\right.
\end{equation}
Note that the radius-like parameter $s$ appears in the scalar ``stretching"
function $R(s,u,v)$. 
Thus, if $R(s,u,v)=0$, then the corresponding point lies on the magnetic axis, 
while if $R(s,u,v)=1$, then the corresponding point lies on the outer surface. 
For fixed value of $s$ between $0$ and $1$, the Garabedian representation 
describes a set of nested surfaces of genus $\mathfrak g = 1$ around the 
magnetic axis - each in the form of the desired parametrization 
\eqref{eq:defX}.
We set $r_{0}(v) = 4.8+0.1 \cos v$,
$z_{0}(v) = 0.1 \sin v$,
$ s = 0.8$,
$R(s,u,v) = s(1+0.01(1-s)\cos u \sin v)$ and the 
coefficients $\Delta_{mn}$ as in Table \ref{fig:Deltas}.
This yields the surface shown in Fig. \ref{fig:geom}.
\begin{table}
\begin{center}
\begin{tabular}[center]{|c||c|c|c|}
\hline
m\textbackslash n & -1 & 0 & 1 \\\hline  \hline
-1 & 0.17 & 0.11&0 \\\hline
0 & 0 & 1 & 0.07 \\\hline
1 & 0 & 4.5 & 0 \\\hline
2 & 0 & -0.25 & - 0.45 \\\hline
\end{tabular}
\end{center}
\caption{Parameters $\Delta_{mn}$ describing the surface.}
\label{fig:Deltas}
\end{table}

\begin{figure}
\begin{center}
\includegraphics[width=.3\textwidth]{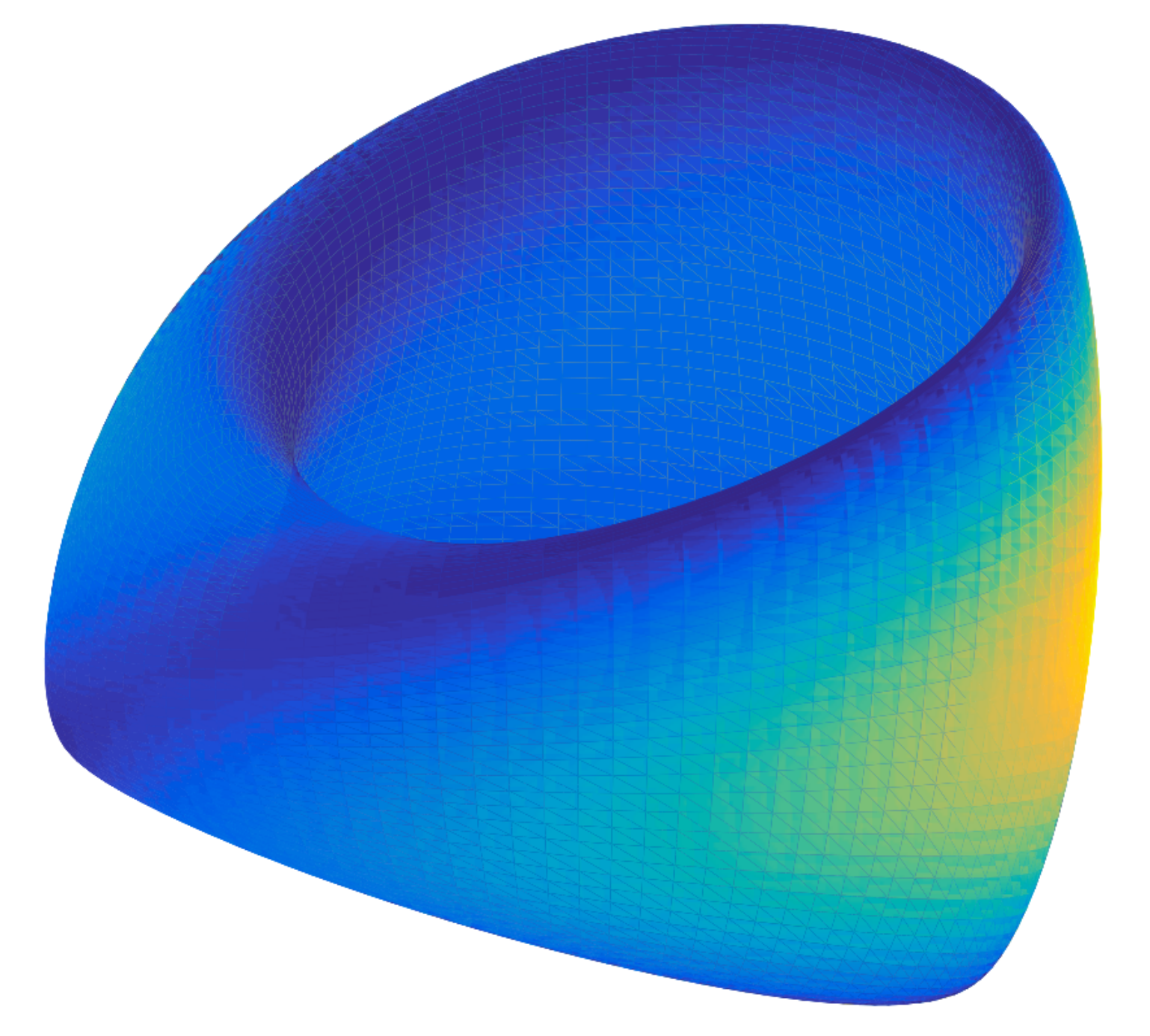}
\includegraphics[width=.35\textwidth]{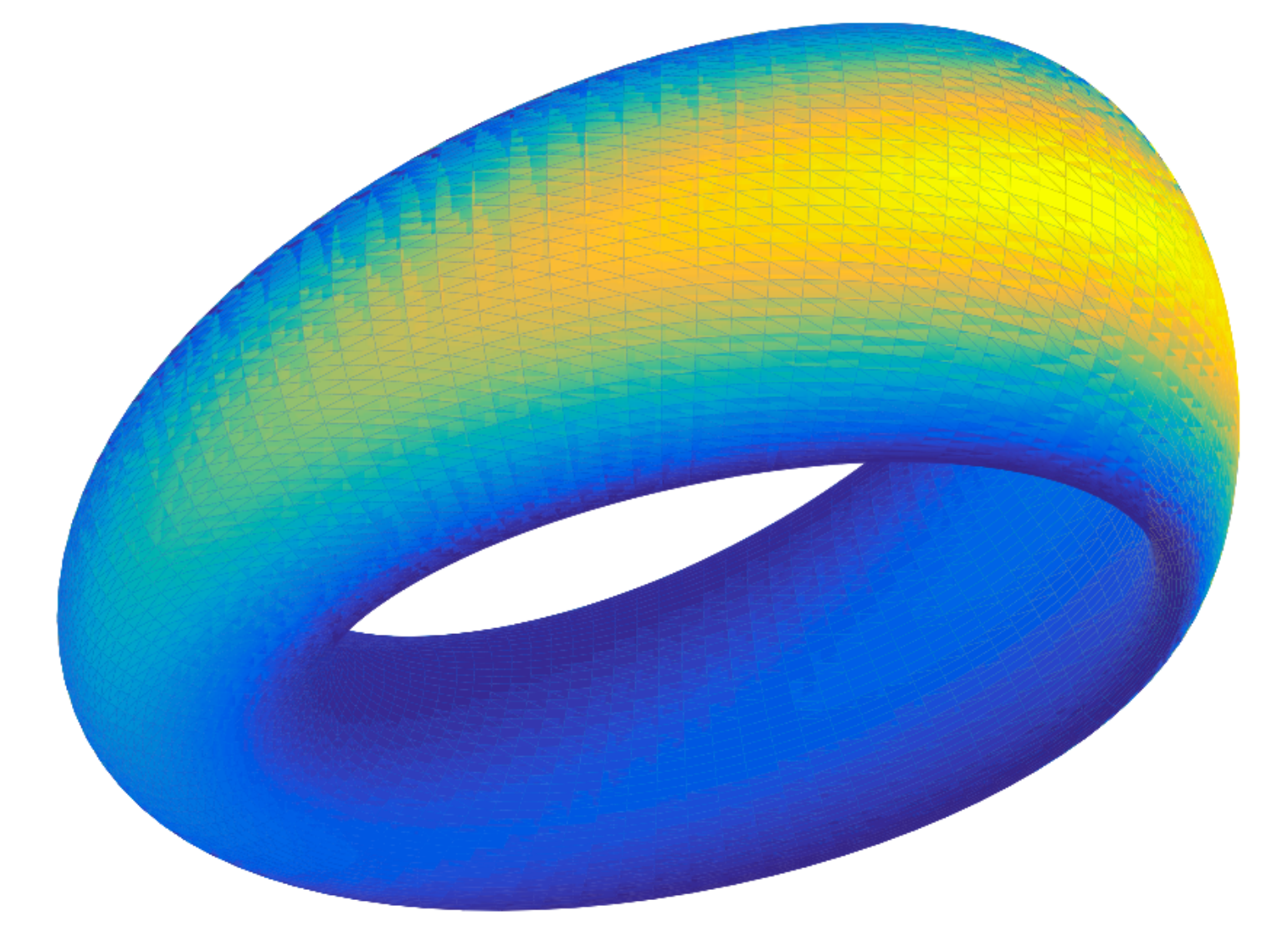}
\end{center}
\caption{Two views of the genus one surface used in our numerical examples.}
\label{fig:geom}
\end{figure}
		
\subsection{The Laplace-Beltrami equation}
In orer to test our Laplace-Beltrami solver, 
we consider the smooth function
$$\psi(u,v) = e^{\cos u+\sin v} + e^{\cos(\kappa(u-v))},$$
and the corresponding mean zero function:
\[ \psi_0 = \psi - 
\frac{\langle \psi,e\rangle_\Gamma}{\langle e ,e\rangle_\Gamma} e \, .
\]
We then compute the corresponding right hand side 
as $b(u,v)=\Delta_\Gamma \psi_0(u,v)$ with
$\kappa = 12$ and solve 
$$
\Delta_{\Gamma,h} \phi = b,
$$
following the method presented in section \ref{surflap}.
Table \ref{tab:LBcv} shows the corresponding convergence results. 
The error is defined as $\| \phi - \psi_0 \|_2$.

\begin{table}
\begin{center}
\begin{tabular}[center]{|c||c|c|c|}
\hline
 N &  Iterations & Time & Error \\
\hline 
         47 &         99 &   0.23E+01 &   0.12E+02 \\
         95 &        125 &   0.97E+01 &   0.36E-02 \\
        191 &        136 &   0.91E+02 &   0.11E-06 \\
        383 &        121 &   0.54E+03 &   0.42E-12 \\
        767 &        114 &   0.86E+03 &   0.44E-12 \\
\hline
\end{tabular}
\end{center}
\caption{Numerical results for the Laplace-Beltrami solver. 
$N$ denotes the number of points used in the uniform discretization in each
component direction (for a total of $N^2$ grid points).
Also shown are the number of BiCG-stab iterations 
\cite{bicgstab}, the solution time and the
$L^2$ error as a function of $N$. }
\label{tab:LBcv}
\end{table}

\subsection{Harmonic vector fields}

We next compute the harmonic vector fields on the surface,
following the method described in section \ref{harmvec}, 
Table \ref{tab:HVFcv} displays the corresponding convergence results. 
\begin{table}
\begin{center}
\begin{tabular}[center]{|c||c|c|c|c|}
\hline
 N & Iterations & Time & Div error & Curl error \\
\hline
         11 &         85 &   0.13E+00 &   0.24E-02 &   0.27E-02\\
         23 &        115 &   0.86E+00 &   0.18E-04 &   0.18E-04\\
         47 &        112 &   0.40E+01 &   0.92E-08 &   0.93E-08\\
         95 &        129 &   0.15E+02 &   0.14E-11 &   0.31E-11\\
        191 &        154 &   0.17E+03 &   0.13E-11 &   0.15E-11\\
        383 &        163 &   0.12E+04 &   0.43E-11 &   0.45E-11\\
        767 &        189 &   0.20E+04 &   0.47E-10 &   0.48E-10\\
\hline
\end{tabular}
\end{center}
\caption{Numerical results for the computation of harmonic vector fields. 
$N$ denotes the number of points used in the uniform discretization in each
component direction (for a total of $N^2$ grid points).
Also shown are the number of BiCG-stab iterations 
\cite{bicgstab}, the solution time and the
$L^2$ error in the surface divergence and surface curl as a function of $N$.}
\label{tab:HVFcv}
\end{table}

\subsection{The Hodge decomposition}

Finally, we consider a vector field in the ambient three-dimensional space
given by
$\mathbf j (\mathbf x)=  \nabla \left(\sin {\imath \kappa \mathbf k \cdot \mathbf x}\right) + \mathbf n \times\nabla \left(\sin {\imath \kappa \mathbf k \cdot \mathbf x}\right) $  
with $\kappa=10$ and $\mathbf k = (0,0,1)$. We define its tangential component $\mathbf j_T$ on the stellarator surface, and
compute its Hodge decomposition using the method presented in section \ref{hodge}. 
This requires solving two Laplace-Beltrami equations to compute the 
curl-free component ($\nabla_\Gamma \alpha$) and the divergence free  component
($\bn \times \nabla_\Gamma \beta$). It also requires
a basis for the harmonic vector fields to determine the harmonic 
component $\bj_H$.
Table \ref{tab:HDcv} presents the corresponding convergence results. 
The error is computed using the $L_2$ norm induced by the inner product
for tangential vector fields in \eqref{l2vecdef}, which we denote by 
$\| \mathbf j_{T} - \nabla_\Gamma \alpha - \bn \times \nabla_\Gamma \beta
- \bj_H \|_2$.
The vector field and its various components are shown in Figure \ref{fig:HD}.

\begin{table}
\begin{center}
\begin{tabular}[center]{|c||c|c|c|}
\hline
N & Iterations & Time & Error \\\hline 
         47 &        112 &   0.12E+02 &   0.21E+00 \\
         95 &        129 &   0.62E+02 &   0.26E-05 \\
        191 &        154 &   0.51E+03 &   0.25E-11 \\
        383 &        163 &   0.36E+04 &   0.42E-13 \\
        767 &        189 &   0.10E+05 &   0.38E-13 \\
        \hline
\end{tabular}
\end{center}
\caption{Numerical results for the Hodge decomposition. 
$N$ denotes the number of points used in the uniform discretization in each
component direction (for a total of $N^2$ grid points).
Also shown are the number of BiCG-stab iterations 
\cite{bicgstab}, the time to find a basis for the harmonic vector fields and
solve the two Laplace-Beltrami equations, and the
$L^2$ error in the reconstructed vector field as a function of $N$.}
\label{tab:HDcv}
\end{table}


\begin{figure}
\begin{center}
\includegraphics[width=.49\textwidth,trim={4cm 9cm 4cm 9cm},clip]{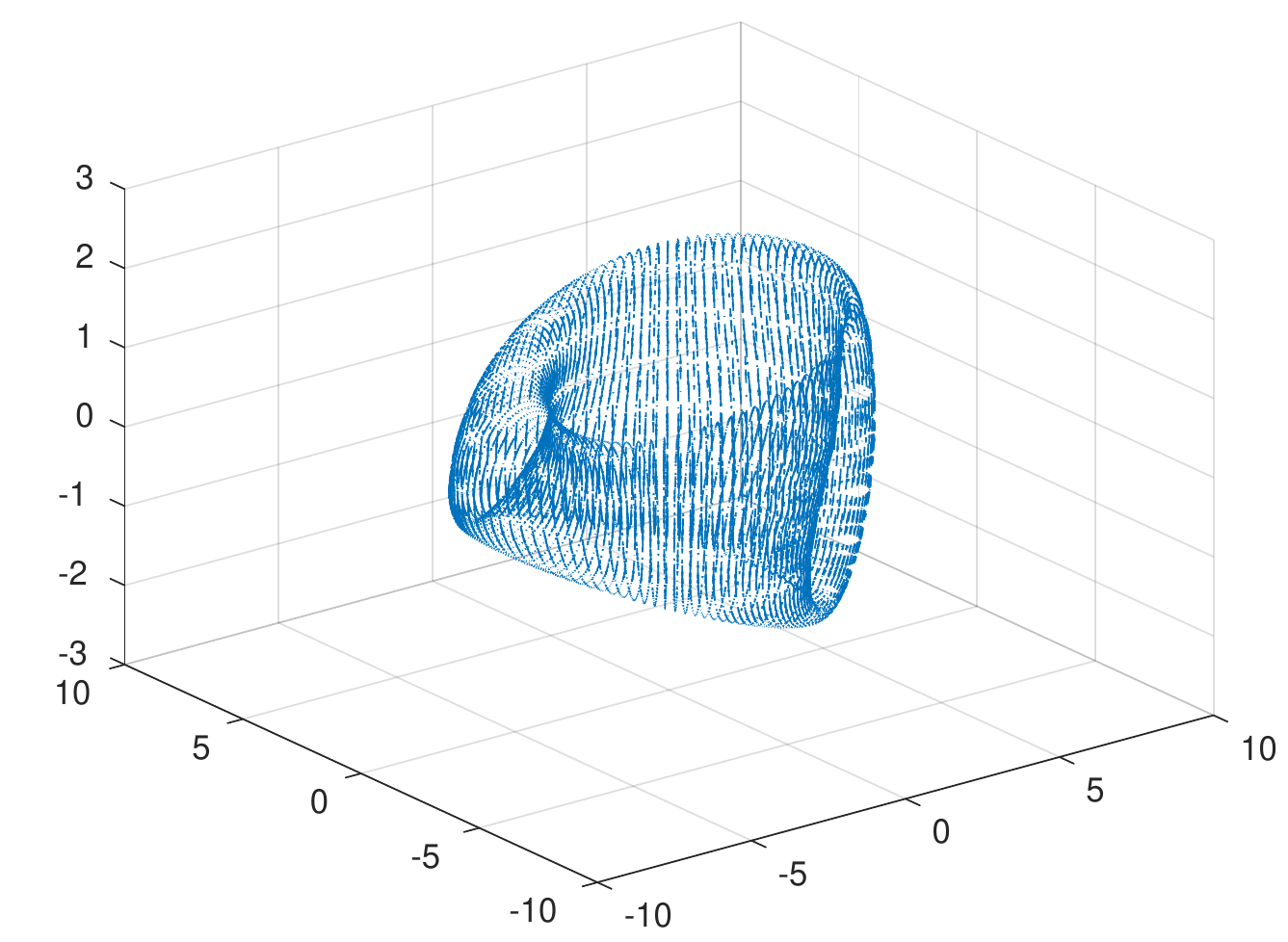}
\includegraphics[width=.49\textwidth,trim={4cm 9cm 4cm 9cm},clip]{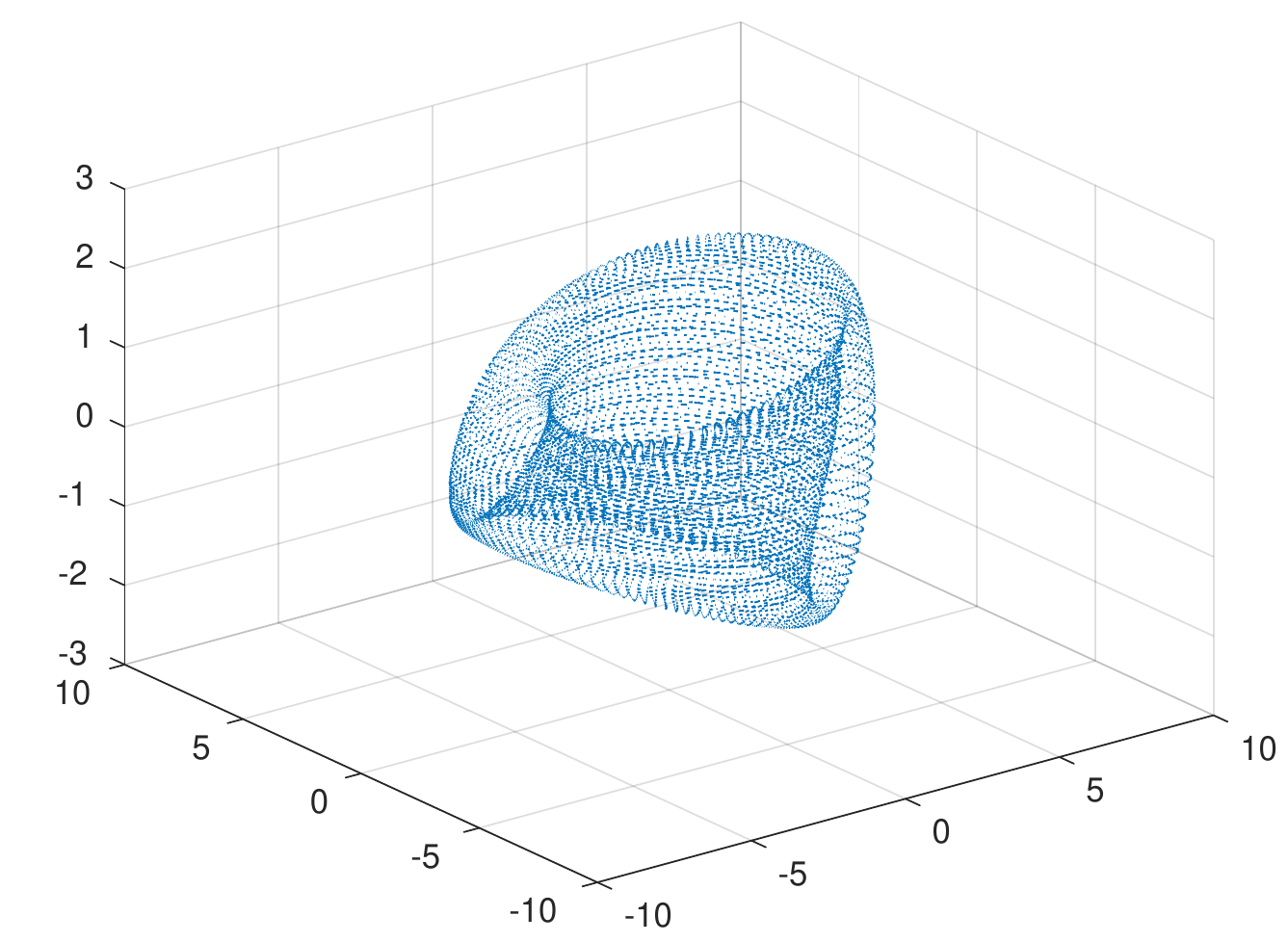}
\end{center}
\begin{center}
\includegraphics[width=.49\textwidth,trim={4cm 9cm 4cm 9cm},clip]{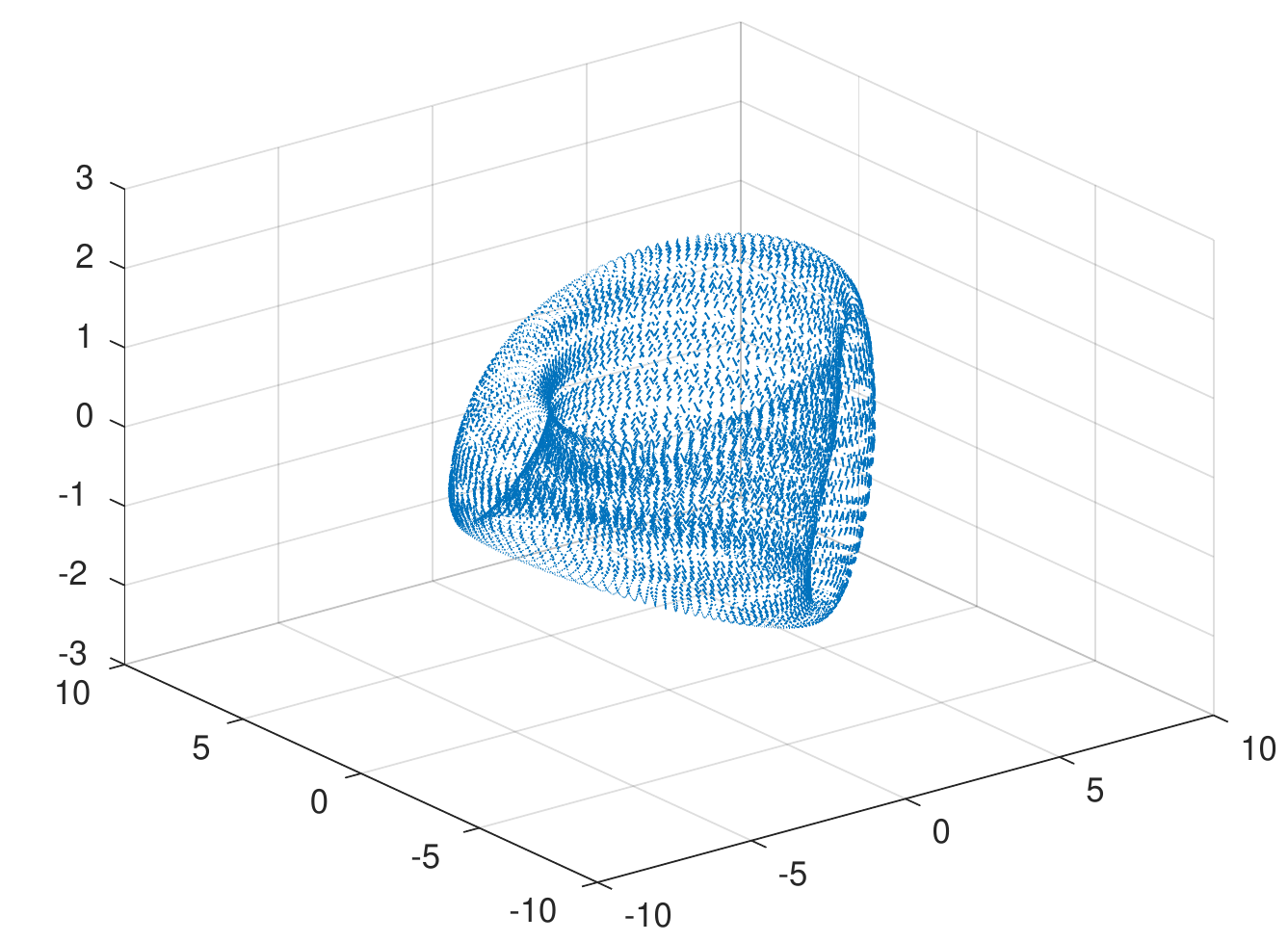}
\end{center}
\caption{The various components of the Hodge decomposition of the vector field $\bj(\mathbf x)=  \nabla \left(\sin {\imath \kappa \mathbf k \cdot \mathbf x}\right) + \mathbf n \times\nabla \left(\sin {\imath \kappa \mathbf k \cdot \mathbf x}\right)$ : its curl-free component $\bn \times \nabla_\Gamma \beta$ (top left),  
 its divergence free component $\nabla_\Gamma \alpha$ (top right), 
and its harmonic component $\bj_H$ (bottom). The vector fields are rescaled for clarity.}
\label{fig:HD}
\end{figure}

\section{Conclusions} \label{discussion} 

We have presented a simple but effective pseudo-spectral method for solving
the Laplace-Beltrami equation on surfaces of genus one, and for determining a 
basis for surface harmonic vector fields. This permits the fast and 
accurate computation of the full Hodge decomposition of a given smooth
tangential vector field. Our method relies on the FFT and a ``flat torus" 
preconditioner, requiring $O(N^2 \log N)$ work where $N^2$ is the number of 
points in the surface representation.

There are several drawbacks to our scheme: it relies on a global
parametrization, it is inherently non-adaptive, and it is limited in its 
present form to surfaces of genus one. We are currently exploring variants of
our approach that will be able to handle these generalizations.

\bibliographystyle{abbrv}
\bibliography{ImbertgerardGreengardHVFArxiv}

\end{document}